\documentclass[12pt]{amsart}
\usepackage{amscd}
\usepackage{amsmath}
\usepackage{verbatim}

\usepackage[cmtip,arrow]{xy}

\usepackage{pb-diagram,pb-xy}

\usepackage{amssymb}

\numberwithin{equation}{section}

\newtheorem{thm}[equation]{Theorem}
\newtheorem{prop}[equation]{Proposition}

\newtheorem{lemma}[equation]{Lemma}
\newtheorem{cor}[equation]{Corollary}

\theoremstyle{definition}
\newtheorem{rem}[equation]{Remark}
\newtheorem{example}[equation]{Example}
\newtheorem{dfn}[equation]{Definition}

\newcommand{\onto}{\rightarrow\!\!\rightarrow}
\newcommand{\codim}{\operatorname{codim}}

\newcommand{\Br}{\mathop{\mathrm{Br}}}

\renewcommand{\Im}{\mathop{\mathrm{Im}}}

\newcommand{\CH}{\mathop{\mathrm{CH}}\nolimits}

\newcommand{\BCH}{\mathop{\overline{\mathrm{CH}}}\nolimits}

\newcommand{\Ch}{\mathop{\mathrm{Ch}}\nolimits}
\newcommand{\BCh}{\mathop{\overline{\mathrm{Ch}}}\nolimits}

\newcommand{\cd}{\mathop{\mathrm{cdim}}\nolimits}

\newcommand{\inc}{\operatorname{\mathit{in}}}

\newcommand{\Z}{\mathbb{Z}}
\newcommand{\F}{\mathbb{F}}
\newcommand{\A}{\mathbb{A}}

\newcommand{\HH}{\mathbb{H}}

\newcommand{\Gm}{\mathbb{G}_{\mathrm{m}}}
\newcommand{\Ga}{\mathbb{G}_{\mathrm{a}}}
\newcommand{\G}{\mathbb{G}}
\newcommand{\cF}{\mathcal F}
\newcommand{\cO}{\mathcal O}
\newcommand{\cE}{\mathcal E}
\newcommand{\cT}{\mathcal T}
\newcommand{\cA}{\mathcal A}

\newcommand{\cS}{\mathcal S}

\newcommand{\Mor}{\operatorname{Mor}}
\newcommand{\Spec}{\operatorname{Spec}}
\newcommand{\End}{\operatorname{End}}
\newcommand{\Hom}{\operatorname{Hom}}

\newcommand{\PS}{\mathbb{P}}

\newcommand{\Sum}{\operatornamewithlimits{\textstyle\sum}}
\newcommand{\Oplus}{\operatornamewithlimits{\textstyle\bigoplus}}

\newcommand{\disc}{\operatorname{disc}}

\newcommand{\<}{\left<}
\renewcommand{\>}{\right>}

\newcommand{\compose}{\circ}

\newcommand{\Ker}{\operatorname{Ker}}

\renewcommand{\phi}{\varphi}

\newcommand{\WR}{\mathcal{R}}

\newcommand{\Fsep}{F_{\mathrm{sep}}}

\usepackage[hypertex]{hyperref}

\title
{Unitary grassmannians}


\author
{Nikita A. Karpenko}

\address
{
Universit\'e Pierre et Marie Curie\\
Institut de Math\'ematiques de Jussieu\\
Paris\\
FRANCE}

\address
{{\it Web page:}
{\tt www.math.jussieu.fr/\~{ }karpenko}}

\email {karpenko {\it at} math.jussieu.fr}

\date
{9 October 2011. Revised: 15 February 2012}

\begin{document}

\begin{abstract}
We study projective homogeneous varieties under an action of
a projective unitary group (of outer type).
We are especially interested in the case of
(unitary) grassmannians of totally isotropic subspaces of a hermitian form over a field, the main
result saying that these grassmannians are $2$-incompressible if the hermitian form
is generic.
Applications to orthogonal grassmannians are provided.
\end{abstract}

\maketitle

\tableofcontents

\section
{Introduction}

In this paper, we study Chow rings, Chow motives, and incompressibility properties of
projective homogeneous varieties under an action of an adjoint absolutely simple
affine algebraic group of type $\cA$ over an arbitrary field $F$.
Such a group can be realized as the $F$-group of $K$-automorphisms of an
Azumaya $K$-algebra with an $F$-linear unitary involution, where $K$ is a
quadratic \'etale $F$-algebra.
The case where $K$ is split is pretty much studied in the literature (the varieties arising in that case are the generalized
Severi-Brauer varieties and, more generally, varieties of flags of right ideals of a central simple $F$-algebra).
In the present paper we are concentrated on the opposite case: $K$ is a
field, but the central simple $K$-algebra is split, i.e., isomorphic to
the algebra of endomorphisms of a finite-dimensional vector space $V$ over $K$.
The involution on the algebra in this case is adjoint to some
non-degenerate $K/F$-hermitian form on $V$.

Let $K/F$ be a separable quadratic field extension, $V$ a vector space
over $K$ endowed with a non-degenerate $K/F$-hermitian form $h$.
By a Jacobson theorem (cf. Corollary \ref{h i q}),
this classical object is completely determined by an even more classical
one:
the quadratic form $v\mapsto h(v,v)$
on $V$ considered this time as a vector space over $F$.
Moreover, the quadratic forms arising this way from $K/F$-hermitian forms, are easily described as the
tensor products of a non-degenerate (diagonal) bilinear form by the norm form of $K/F$ (which is a non-degenerate binary
quadratic form, the quadratic form associated to
the hermitian form $\<1\>$ on $K$).

Although this well-known and elementary observation shows that the study of hermitian
forms is equivalent to the study of the ``binary divisible'' quadratic forms,
it does not show that the hermitian forms
are not interesting because they are just a more complicated replacement for a well-understood thing.
Quite the contrary,
first of all this observation shows that the binary divisible quadratic
forms form an important class of quadratic forms.
%
On the other hand,
this observation provides a hope that the world around hermitian forms might give
additional tools to study such quadratic forms. And indeed, it turns out
that the projective homogeneous varieties related to hermitian forms (in particular, the unitary grassmannians)
are more suitable
(in particular, more ``economical'')
geometric objects related to the binary divisible quadratic forms if
compared with the orthogonal grassmannians
-- the geometric objects
commonly attached to quadratic forms.
The present paper demonstrates and exploits these phenomena.

The main result of the paper is Theorem \ref{main} saying that the unitary
grassmannians associated to a generic hermitian form are
$2$-incompressible \footnote{We refer to \cite{canondim} for the definition and main properties of
{\em incompressibility} and, more generally, of {\em canonical
dimension} of (projective homogeneous) varieties.}
and that essential parts of their motives are
indecomposable.
The main application (which was the initial motivation) is Corollary
\ref{J(generic)},
computing the $J$-invariant of the quadratic form attached to the
generic hermitian form and therefore the minimal value of
$J$-invariant of a quadratic form arising from a hermitian one (of a fixed dimension).
To avoid confusion, let us note that we are using the original notion of $J$-invariant, introduced by
A. Vishik in \cite[Definition 5.11]{MR2148072}, and not the modified
version of \cite{EKM}, see Section \ref{Applications to quadratic forms}
for details.

Another striking application (in the opposite direction: from quadratic
forms to hermitian ones) is the computation of the modulo $2$ reduced Chow
ring of a unitary grassmannians in terms of the Chow ring of the
orthogonal grassmannians made in Corollary \ref{isochow} (for odd-dimensional hermitian forms see
Example \ref{9.11}).
This computation is particularly interesting in the hyperbolic case because the Chow ring of orthogonal grassmannians of hyperbolic forms
is very well studied and described, see e.g. \cite{Kresch} and \cite{Vishik-u-invariant}.

We start in Section \ref{An elementary computation} with a completely
elementary computation of the subring of the invariant elements under a
permutation involution on a polynomial ring.
It becomes clear already at this point that one can make the life much more agreeable
working modulo the norms.
This principle will be constantly present in the sequel.

The reason of making the above computation becomes clear in Section \ref{Some equivariant Chow groups of the point}
where we are interested in the $T$-equivariant Chow ring of the point $\Spec F$ for
certain quasi-split torus $T$: this Chow ring is identified with the
subring of Section \ref{An elementary computation}.
Our base field $F$ appearing for the first time (after Introduction) at this point,
is (and remains until the end of the paper) of arbitrary characteristic.

In Section \ref{Generic varieties of complete flags}, the main objects of our
study begin to appear.
The torus $T$ (or $\Gm\times T$ -- depending on parity of $n$)
considered in Section \ref{Some equivariant Chow groups of the point},
turns out to be a maximal torus in a Borel subgroup $B$ of a non-split
quasi-split adjoint absolutely simple affine algebraic group $G$ of
type $\cA_n$.
Consequently,
as already noticed in \cite{MR2258262},
the $T$-equivariant Chow group $\CH_T(\Spec F)$ appears in
a computation of the reduced Chow group of the complete flag variety $E/B$
given by a {\em generic} (in the sense of \cite{MR1999383}) $G$-torsor $E$ (where $B\subset G$ is a Borel subgroup).
Having shown that the ring $\CH_T(\Spec F)$ is generated modulo the norms
by elements of codimension $\leq2$, we get the same statement for the
reduced Chow group $\BCH(E/B)$ (Proposition \ref{E/B}).
Killing in a generic way a (certain) Tits algebra, we come to the complete flag variety of a
(generic)
hermitian form while keeping the above property (Corollary \ref{corE/B}).

Next we would like to get from the complete flag variety to the maximal
grassmannian.
Clearly, this can be achieved by projecting complete flags to their
maximal component.
But unlikely the case of the orthogonal group, such a projection here is not
a chain of projective bundles.
It is a chain of quadratic Weil transfers of projective bundles.
So, we make some general study of quadratic Weil transfers of projective bundles
in Section \ref{Weil transfer of projective bundles}.\footnote{This is
already needed in the previous section at the moment we are generically
killing the Tits algebra, because this is a central simple algebra over an
extension of the base field.}
Once again, working modulo norms makes it possible to get a very simple
(and very similar to the classical case of a projective bundle)
description of the Chow ring.

We apply this description in Section \ref{Generic maximal grassmannians}
getting a description of the reduced Chow group of a generic maximal
unitary grassmannian.
We still get the generation by codimension $\leq2$, but at the same time
the components of codimension $1$ and $2$ vanish, showing that the
reduced Chow group modulo norms is trivial in positive codimensions (Proposition \ref{1.9}).
This description actually shows that such a grassmannian is
$2$-incompressible.
More than this, an essential part of its Chow motive (with coefficient in $\F_2$)
is indecomposable.
But the last two statements belong already to Section \ref{Generic
grassmannians} and are valid for all (generic) grassmannians, not only for
the maximal one.
Before we come to them,
we need a definition of the essential part and some general observations on isotropic
grassmannians,  made in Section \ref{Isotropic grassmannians}.

The notion of the essential motive of a unitary grassmannian,
introduced in Definition \ref{def-essential}, gives one more justification
for working "modulo norms": the reduced Chow group of the essential motive is
identified with the reduced Chow group of the variety modulo the norms.

The proof of Theorem \ref{main}, the main result of the paper, closely
follows \cite{gog} (see also \cite{sgog}), where a similar result is proved for the orthogonal
grassmannians.
The difference is in the case of the maximal grassmannian (the induction
base of the proof), which constitutes the main part of the present paper
while in \cite{gog} (and in \cite{sgog}) it was served by a reference.

The final Section \ref{Applications to quadratic forms} establishes and
exploits the connection with the quadratic forms.
It starts by (I suppose) well-known Lemma \ref{i2i} and classically known Corollary
\ref{h i q}.
Corollary \ref{X i Y} is just a translation of Lemma \ref{i2i} to the
language of upper motives.
The first interesting result is Corollary \ref{J(generic)} (discussed
above already)
computing the
minimal possible value of the $J$-invariant of a quadratic form
split by a given separable quadratic field extension.
This result was not accessible by methods of quadratic forms themselves,
but transferring the problem to the world of hermitian forms, we come to a
problem about a generic hermitian form (without any additional condition)
which we can solve because such a hermitian form can be obtained
from a generic torsor with a help of generic killing of
a Tits algebra.\footnote{Actually, we use the name {\em generic hermitian form} in a more concrete and narrow sense
(defined before Corollary \ref{corE/B})
in the main body of the article.
At this (non-formal) point, speaking about ``generic'' we mean a versality property in the spirit of
\cite[Definition 5.1(2)]{MR1999384}.}

The motivic indecomposability which we have in the generic case by Theorem
\ref{main} gives an isomorphism of motives of Proposition \ref{dvaMr} (in generic and therefore in general case)
which implies an isomorphism of Chow groups of Corollary \ref{isochow}.
This isomorphism of Chow groups is interesting even in the quasi-split case:
for instance, Examples \ref{primerchik} and \ref{9.11} provide a description by generators and relations for the Chow group of a
quasi-split maximal unitary grassmannian (in terms of such description known in the orthogonal case, \cite[\S86]{EKM}).

\bigskip
\noindent
{\sc Acknowledgements.}
I am grateful to Vladimir Chernousov, Alexander Merkurjev, Burt Totaro, Alexander Vishik, and Maksim
Zhykhovich for useful discussions and advices.
Special thanks go to the referee for careful reading and to Alexander Merkurjev for Remark \ref{asmrem} and for
noticing that the computations of Section \ref{Some equivariant Chow groups of the
point} are valid for arbitrary \'etale algebras (or arbitrary quasi-split tori).

\section
{Invariants under a permutation involution}
\label{An elementary computation}

Let $R$ be a commutative associative unital ring.
For any integer $r\geq0$, let us consider
the polynomial
$R$-algebra
in $2r$ variables
$S:=R[a_1,b_1,\dots,a_r,b_r]$.
Let $\sigma$ be the involution (i.e., a self-inverse automorphism)
of the $R$-algebra $S$
interchanging for
every $i=1,\dots,r$ the variables $a_i$ and $b_i$.
We write $S^\sigma\subset S$ for the
subalgebra
of $\sigma$-invariant elements.

The subset $(1+\sigma)S:=\{s+\sigma(s)\;|\;\;s\in S\}\subset S^\sigma$, the image of the
norm homomorphism
$N:S\to S$, $s\mapsto s+\sigma(s)$ (this is a homomorphism of $R$-modules), is contained in $S^\sigma$ and
is an ideal of
$S^\sigma$.

\begin{lemma}
\label{S}
The quotient $S^\sigma/(1+\sigma)S$ is generated as $R$-algebra by the classes of
the products $a_ib_i$, $i=1,\dots,r$.
\end{lemma}

\begin{proof}
Induction on $r$.
The statement is trivial for $r=0$.

To pass from $r-1$ to $r\geq1$, let $S'$ be the $R$-subalgebra
$R[a_1,b_1,\dots,a_{r-1},b_{r-1}]$ of $S$.
The $S'$-module $S$ is free with the basis $\{a_r^ib_r^j\}_{i,j\geq0}$ so
that
an element $\Sum_{i,j\geq0}\alpha_{ij}a_r^ib_r^j\in S$ with $\alpha_{ij}\in S'$
is $\sigma$-invariant if and only
if $\sigma(\alpha_{ij})=\alpha_{ji}$ for any $i,j$.
It follows that the ${S'}^\sigma/(1+\sigma)S'$-algebra $S^\sigma/(1+\sigma)S$ is generated
by the class of $a_rb_r$.
\end{proof}

\begin{rem}
In general, the $R$-algebra $S^\sigma$ itself
is not generated by polynomials of degree
$\leq2$.
Taking for instance $R=\Z$ and $r=3$, we have a $\sigma$-invariant integral polynomial $P:=a_1a_2a_3+b_1b_2b_3$
which cannot be written as a polynomial in the $\sigma$-invariant polynomials of degree $\leq2$ in
$a_1,b_1,a_2,b_2,a_3,b_3$.
Indeed,
let us refer as {\em elementary polynomial} to any $\sigma$-invariant monomial (with coefficient $1$)
and to any polynomial of the form $M+\sigma(M)$, where $M$ is a
non-$\sigma$-invariant monomial (with coefficient $1$).
Any degree $3$ product of elementary
polynomials of degree $1$ or $2$ either has no
monomial with repeated indices of variables (only has monomials like
$a_1a_2b_3$, but no monomials like $a_1b_1a_3$ or $a_1^2a_3$) or consists
only of monomials with repeated indices.
If $P$ is a linear combination of such products with coefficients $\pm1$
(which is the case if $P$ is a polynomial in the $\sigma$-invariant polynomials of degree $\leq2$),
then the part of the combination given by the products of the second type
is $0$.
Each product in the remaining part of the decomposition, as any product of the first type, has the following
property: it is a sum of {\em even} number of degree $3$
elementary polynomials.
In particular, $P$ is a linear combination with coefficients $\pm1$ of an even number of
elementary polynomials.
Any cancelation keeps the parity of this number.
However $P$ is just one elementary polynomial.

On the other hand, $S^\sigma$ {\em is} generated by its degree $1$
and $2$ components as far as $2$ is invertible in $R$:
this can be shown by induction using the identity
$$
2(fgh+f'g'h')=
(fg+f'g')(h+h')-(fh'+f'h)(g+g')+(gh+g'h')(f+f').
$$
\end{rem}

\section
{Some equivariant Chow groups of the point}
\label{Some equivariant Chow groups of the point}

Let $F$ be a field.
We recall the computation of the $\Gm$-equivariant Chow ring
$$
\CH_{\Gm}(\Spec F),
$$
c.f. \cite[Lemma 4]{MR1487222}.
For generalities on equivariant Chow groups and rings, we refer to \cite{MR1614555}
and \cite{MR1953530}.

The $F$-torus $\Gm$ acts by multiplication on the affine line $\A^1=\A^1_F$ over $F$.
For any integer $l\geq1$, we consider the diagonal action of
$\Gm$ on $\A^l$.
By the homotopy invariance of the equivariant Chow groups,
the pull-back ring homomorphism $\CH_{\Gm}(\Spec F)\to\CH_{\Gm}(\A^l)$ is an
isomorphism.
By the localization sequence,
the pull-back ring homomorphism $\CH_{\Gm}(\A^l)\to\CH_{\Gm}(U_l)$ to the
$\Gm$-invariant open subset
$U_l:=\A^l\setminus\{0\}$ is bijective in codimensions $<l$
(because the complement of $U_l$ in $\A^l$ is of codimension $l$).
It follows that the ring homomorphism $\CH_{\Gm}(\Spec F)\to\CH_{\Gm}(U_l)$
is bijective in codimensions $<l$.
Since $\Gm$ acts freely on $U_l$ and $U_l/\Gm=\PS^{l-1}$, the
ring $\CH_{\Gm}(U_l)$ is identified with the usual Chow ring $\CH(\PS^{l-1})$.
It follows that the ring $\CH_{\Gm}(\Spec F)$ is freely generated by the
element
$h\in\CH^1_{\Gm}(\Spec F)$ corresponding to the class
of a rational point in $\CH^1(\PS^1)$ under the isomorphism
$\CH^1_{\Gm}(\Spec F)\to\CH^1(\PS^1)$.

Now we fix an \'etale $F$-algebra $L$
(see \cite[\S18A]{MR1632779} for generalities on \'etale algebras)
and
consider the Weil transfer $\WR_{L/F}\Gm$ (see \cite[20.5]{MR1632779} where it is called {\em corestriction})
with respect to $L/F$ of the $L$-torus $\Gm$.
The $F$-torus $\WR_{L/F}\Gm$ is quasi-split, and any quasi-split $F$-torus is
isomorphic to $\WR_{L/F}\Gm$ for an appropriate choice of $L$.
We would like to describe the graded (by codimension) ring $\CH_{\WR_{L/F}\Gm}(\Spec F)$.

We fix a separable closure $\Fsep$ of $F$, write
$\Gamma$ for the Galois group of $\Fsep/F$, and $X$ for
the $\Gamma$-set of the $F$-algebra homomorphisms $L\to\Fsep$.
We consider the following graded rings:
the polynomial ring $\Z[X]$ in variables indexed by the
elements of $X$ and its subring $\Z[X]^{\Gamma}$ of $\Gamma$-invariant
elements.


Let us do a computation for $\CH_{\WR_{L/F}\Gm}(\Spec F)$ similar to the above computation
of $\CH_{\Gm}(\Spec F)$.
The torus $\WR_{L/F}\Gm$ acts by multiplication on the affine $F$-space $\WR_{L/F}\A^1$
(of dimension $\#X=\dim_FL$).
For any integer $l\geq1$, we consider the diagonal action of
$\WR_{L/F}\Gm$ on the affine $F$-space $\WR_{L/F}(\A^l)=(\WR_{L/F}\A^1)^l$.
The action on the open $\WR_{L/F}\Gm$-invariant subset $U_l:=\WR_{L/F}(\A^l\setminus\{0\})$
is free, and $U_l/\WR_{L/F}\Gm=\WR_{L/F}\PS^{l-1}$.
We get a ring homomorphism $\CH_{\WR_{L/F}\Gm}(\Spec F)\to\CH(\WR_{L/F}\PS^{l-1})$
which is bijective in codimensions $<l$ (because the complement of $U_l$ in $\WR_{L/F}\A^l$ is of
codimension $l$).

Since the integral Chow motive
(for generalities on motives wee \cite[Chapter XII]{EKM})
of the variety $\WR_{L/F}\PS^{l-1}$ is a sum of
shifts of the motives of $\Spec K$ for certain subfields $K\subset\Fsep$ finite over $F$
(see \cite{MR1809664} or
\cite{outer}), the ring homomorphism $\CH(\WR_{L/F}\PS^{l-1})\to\CH(\WR_{L/F}\PS^{l-1})_{\Fsep}$
identifies the ring $\CH(\WR_{L/F}\PS^{l-1})$ with the subring of $\Gamma$-invariant elements in
$\CH(\WR_{L/F}\PS^{l-1})_{\Fsep}$.
It remains to identify the $\Fsep$-variety $(\WR_{L/F}\PS^{l-1})_{\Fsep}$ with the
product of copies of
$\PS^{l-1}$ indexed by the elements of the set $X$ and therefore to identify
$\CH(\WR_{L/F}\PS^{l-1})_{\Fsep}$ with the tensor product of the indexed by the elements of $X$ copies of
$\CH(\PS^{l-1})$.
Therefore we obtain an isomorphism of graded rings
\begin{equation}
\label{asm}
\CH_{\WR_{L/F}\Gm}(\Spec F)=\Z[X]^{\Gamma}.
\end{equation}

\begin{rem}
Any element $L\to\Fsep$ of $X$ gives a homomorphism of $\Fsep$-algebras
$L_{\Fsep}\to\Fsep$ and therefore a character of the $\Fsep$-torus
$(\WR_{L/F}\Gm)_{\Fsep}$.
The obtained map $X\to {(\WR_{L/F}\Gm)_{\Fsep}}^*$ (to the group of
characters)
induces an isomorphism of graded $\Gamma$-rings $\Z[X]\to
\cS\big({(\WR_{L/F}\Gm)_{\Fsep}}^*\big)$ to the symmetric ring
of the abelian group ${(\WR_{L/F}\Gm)_{\Fsep}}^*$.
Therefore the isomorphism \ref{asm} gives rise to an
(actually independent of the choice of $L$)
isomorphism
$$
\CH_T(\Spec F)=\cS({T_{\Fsep}}^*)^{\Gamma}
$$
(for any quasi-split $F$-torus $T$).
\end{rem}

Let now $K$ be a separable quadratic field extension of $F$.
Applying the isomorphism \ref{asm} to the \'etale $F$-algebra $L=K^{\times r}$,
since $\WR_{L/F}\Gm=(\WR_{K/F}\Gm)^{\times r}$,
we get

\begin{lemma}
\label{ourT}
Let $T:=(\WR_{K/F}\Gm)^{\times r}$ for some $r\geq1$.
There is an isomorphism
of the graded ring $\CH_T(\Spec F)$ with the subring of $\sigma$-invariant
elements of the ring
$$
\Z[a_1,b_1,\dots,a_r,b_r]
$$
(where $\sigma$ is the
ring involution exchanging $a_i$ and $b_i$ for each $i=1,\dots,r$)
such that the image of the norm map $\CH_T(\Spec K)\to\CH_T(\Spec F)$
corresponds to the image of the map
$1+\sigma:\Z[a_1,b_1,\dots,a_r,b_r]\to\Z[a_1,b_1,\dots,a_r,b_r]^\sigma$.
\qed
\end{lemma}

Note that the image of the norm map $\CH_T(\Spec K)\to\CH_T(\Spec F)$
is an ideal of the destination ring.
Applying Lemma \ref{S}, we get

\begin{cor}
\label{cor1}
The ring $\CH_T(\Spec F)$  modulo the ideal of norms from
$\Spec K$ is generated
by elements of codimension $2$.
\qed
\end{cor}

Similarly, applying isomorphism (\ref{asm}) to the \'etale $F$-algebra
$L=F\times K^{\times r}$ and then Lemma \ref{S} with $R$ being a ring integral
polynomials in one variable, we get

\begin{cor}
\label{cor2}
The ring $\CH_{\Gm\times T}(\Spec F)$
(for $T$ as in Lemma \ref{ourT})
modulo the ideal of norms from
$\Spec K$ is generated
by elements of codimension $1$ and $2$.
\qed
\end{cor}

\section
{Generic varieties of complete flags}
\label{Generic varieties of complete flags}

Let $G$ be a quasi-split
semisimple affine algebraic group over a
field $F$.
Let $K/F$ a minimal field extension such that $G_K$ is split.
(This is a finite galois field extension uniquely determined up to an isomorphism by
$G$.
Our main case of interest will be the case where $G$ is non-split adjoint and simple of
type $\cA_{n-1}$, $n\geq2$; in this case the extension $K/F$ is
quadratic.)

Let $P\subset G$ be a parabolic subgroup.

\begin{lemma}
\label{injects}
For any field extension $L/F$, the change of field homomorphism
$\CH(G/P)\to\CH(G_L/P_L)$ is injective.
It is moreover bijective provided that $K\otimes_FL$ is a field.
\end{lemma}

\begin{proof}
By \cite{MR2110630}, the Chow motive (with integer coefficients) of the $F$-variety
$G/P$ decomposes in a finite direct sum of shifts of the motive of the $F$-varieties $\Spec
K'$, where $K'$ runs over intermediate fields of the extension $K/F$.
\end{proof}

\begin{cor}
\label{humphreys}
The action of $G(F)$ on $\CH(G/P)$ is trivial.
\end{cor}

\begin{proof}
By Lemma \ref{injects}, we may assume that $F$ is algebraically closed.
In this case, by \cite[Theorem of \S26.3 in Chapter IX]{MR0396773},
the group $G(F)$ is generated by the subgroups $\G(F)$, where $\G$ runs
over subgroups of $G$ isomorphic to $\Ga$ or $\Gm$.
To finish the proof it suffices to show that for $\G=\Ga$ as well as for $\G=\Gm$,
given an action of $\G$ on a smooth variety $X$ over a field $F$ (which needs not to be algebraically closed
anymore), the induced action of $\G(F)$ on $\CH(X)$ is trivial.

To show this, we proceed similarly to \cite[Proof of Lemma 4.3]{MR1460391}.
Note that the pull-back $\CH(X)\to\CH(\G\times X)$ with
respect to the projection $\G\times X\to X$ is an isomorphism:
for $\G=\Ga$ this is so by homotopy invariance of Chow groups \cite[Theorem
57.13]{EKM};
for $\G=\Gm$ the homotopy invariance together with the localization
\cite[Proposition 57.9]{EKM} give the surjectivity, and for injectivity
consider a section of the projection given by a rational point on $\Gm$.

It follows that the pull-back $\inc_g^*:\CH(\G\times X)\to\CH(X)$ with
respect to the section $X\to\G\times X$ given by an element $g\in\G(F)$
is the inverse of the above isomorphism and, in particular, does not depend on $g$.
On the other hand, the action of $g$ on $\CH(X)$ is the composition
$$
\begin{CD}
\CH(X)@>>>\CH(\G\times X)@>{\inc_g^*}>>\CH(X),
\end{CD}
$$
where the first map is the pull-back with respect to the action morphism $\G\times X\to
X$ of $\G$ on $X$.
Therefore the action of $g$ on $\CH(X)$ coincides with the action of $1$ which is trivial.
\end{proof}

For any $G$-torsor (i.e., a principle $G$-homogenous space) $E$ over $F$, we construct a homomorphism
$$
\CH(E/P)\to\CH(G/P)
$$
as the composition of three ones:
the change of field homomorphism
$$
\CH(E/P)\to\CH(E_{F(E)}/P_{F(E)}),
$$
the pull-back
$$
\CH(E_{F(E)}/P_{F(E)})\to\CH(G_{F(E)}/P_{F(E)})
$$
with respect to the morphism $G_{F(E)}/P_{F(E)}\to E_{F(E)}/P_{F(E)}$
induced by the $G$-equivariant morphism $G_{F(E)}\to E_{F(E)}$ taking the
identity of $G$ to the generic point of $E$, and
the inverse
$$
\CH(G_{F(E)}/P_{F(E)})\to\CH(G/P)
$$
of the change of field homomorphism
$\CH(G/P)\to\CH(G_{F(E)}/P_{F(E)})$
which is an isomorphism by Lemma \ref{injects} because $F$ is algebraically closed in $F(E)$.

More generally, the homomorphism $\CH(E/P)\to\CH(G/P)$ is defined for any
$G$-torsor $E$ over a {\em regular} field extension $\tilde{F}/F$:
since $K\otimes_F\tilde{F}$ is a field, the change of fields homomorphism
$\CH(G/P)\to\CH(G_{\tilde{F}}/P_{\tilde{F}})$ is an isomorphism.

\begin{lemma}
\label{spec}
Let $S$ be a smooth geometrically irreducible $F$-variety and let $E$ be a
$G$-torsor over $S$.
For a point $s\in S(L)$ of $S$ in a regular field extension $L/F$, let
$E_s$ be the $G$-torsor over $L$ given by the fiber of $E\to S$ over $s$.
Then the image of $\CH(E_s/P)\to\CH(G/P)$ contains the image of
$\CH(E_\theta/P)\to\CH(G/P)$, where $\theta\in S(F(S))$ is the generic
point of $S$.
\end{lemma}

\begin{proof}
Let $x\in S$ be the image of the point $\Spec L$ under the morphism $s:\Spec L\to
S$.
Since $x$ is regular, there exists a system of local parameters on $S$ around
$x$.
Therefore the fields $F(S)$ and $F(x)$ are connected by a finite chain of
discrete valuation fields, where each next field is the residue field of
the previous one.
Using the specialization homomorphisms on Chow groups as in \cite[Example
20.3.1]{Fulton}, we get a homomorphism $\CH(E_\theta/P)\to\CH(E_x/P)$.
Composing it with the change of fields homomorphism
$\CH(E_x/P)\to\CH(E_s/P)$, we get a homomorphism $\CH(E_\theta/P)\to\CH(E_s/P)$
which forms a commutative triangle with the homomorphisms to $\CH(G/P)$ and the required inclusion follows.

To prove that the triangle is commutative,
it suffices to prove that the similar triangle based on the specialization homomorphism
$\CH(E_\theta/P)\to\CH(E_x/P)$ commutes in the following situation:
$R$ is a discrete valuation $F$-algebra with the residue field $F$;
$E$ is a $G$-torsor over $R$;
$x$ is the closed and $\theta$ the generic point of $\Spec R$.
Moreover, extending $F$ to $F(E_x)$, we may assume that the torsor $E_x$
is split.
In this case,
by \cite[Proposition 8.1]{Grothendieck62},
the torsor $E_{\hat{R}}$ over
the completion $\hat{R}$ of $R$ is split.
Therefore the triangle with
$\CH(E_{\hat{\theta}}/P)\to\CH(E_{\hat{x}}/P)$, where $\hat{\theta}$ and
$\hat{x}$ are the generic and closed points of $\Spec \hat{R}$,
commutes (the homomorphism $\CH(E/P)\to\CH(G/P)$ induced by any trivialization of a split torsor $E$ is the identity
by Corollary \ref{humphreys}).
Commutativity of the diagram
$$
\begin{CD}
\CH(E_\theta/P)@>>>\CH(E_x/P)\\
@VVV      @VVV\\
\CH(E_{\hat{\theta}}/P)@>>>\CH(E_{\hat{x}}/P)
\end{CD}
$$
finishes the proof.
\end{proof}

Let $B\subset G$ be a Borel subgroup and $T\subset B$ a maximal torus.
The following statement is well-known in the case of split $G$.
In our quasi-split case, the proof is almost the same:

\begin{lemma}
\label{B/T}
For any $G$-torsor $E$,
the pull-back $\CH(E/B)\to\CH(E/T)$ with respect
to the projection $E/T\to E/B$ is an isomorphism.
In particular, the image of $\CH(E/B)\to\CH(G/B)$ is identified with the image
of $\CH(E/T)\to\CH(G/T)$.
\end{lemma}

\begin{proof}
Let $U$ be the unipotent part of $B$.
By \cite[Theorem 10.6 of Chapter III]{MR0251042},
$U$ is a normal subgroup of $G$ possessing a finite increasing chain of normal
subgroups $U_i$ with each successive quotients isomorphic to $\Ga$, and
$B$ is a semidirect product of $U$ by $T$.
Since $H^1(L,\Ga)$ is trivial for any field extension $L/F$,
the fiber over any point of the projection
$$
E/(U_{i-1}\rtimes T)\to E/(U_i\rtimes T)
$$
is isomorphic to an
affine line
so that the pull-back of Chow groups is an isomorphism
by the homotopy invariance of Chow groups, \cite[Theorem 57.13]{EKM}.
\end{proof}

From now on,
let $G$ be a non-split quasi-split adjoint absolutely simple affine algebraic group
over a field $F$ of type $\cA_{n-1}$, $n\geq 2$, becoming split over a separable quadratic field extension
$K/F$.
We set $r:=\lfloor n/2\rfloor$ (the floor of $n/2$).

\begin{lemma}
\label{T}
The group $G$ contains a Borel subgroup $B$ and a maximal torus $T\subset B$
such that
$$
T\simeq
\begin{cases}
\Gm\times(\WR\Gm)^{\times(r-1)}, \text{ if $n$ is even (i.e., $n=2r$);}\\
(\WR\Gm)^{\times r}, \text{ if $n$ is odd (i.e., $n=2r+1$),}
\end{cases}
$$
where $\WR=\WR_{K/F}$ is the Weil transfer with respect to the field extension $K/F$.
\end{lemma}

\begin{proof}
In the case of $n=2r$,
let $h$ be the orthogonal sum of $r$ copies of the $K/F$-hermitian
{\em hyperbolic plane} (the $K/F$-hermitian form on the vector space $K^2$ of the matrix
$\begin{pmatrix}
0&1\\ 1&0
\end{pmatrix}
$).
In the case of $n=2r+1$, we let $h$ be the orthogonal sum of $r$ copies of the
hyperbolic plane and of the $K/F$-hermitian space $\<1\>$.
In both cases, $h$ is a hermitian form on the $K$-vector space $K^n$.
Up to an isomorphism,
$G$ is the $F$-group of automorphisms of the $K$-algebra $\End_K(K^n)$ with
the (unitary) adjoint involution.
We may assume that $G$ is this group of automorphisms.

For $n=2r$, we have a homomorphism $\alpha:(\WR\Gm)^{\times r}\to G$, mapping an $F$-point
$(\lambda_1,\dots,\lambda_r)\in (\WR\Gm)^{\times r}(F)=K^\times\times\dots\times K^\times$
to the automorphism of $\End_KK^n$ which is the conjugation by the
diagonal automorphism $(\lambda_1,\sigma(\lambda_1)^{-1},\dots,\lambda_r,\sigma(\lambda_r)^{-1})$
of $K^n$ preserving $h$,
where $\sigma$ is the non-trivial automorphism of $K/F$.
An $F$-point $(\lambda_1,\dots,\lambda_r)$ is in the kernel of $\alpha$ if
and only if
$\lambda_1=\sigma(\lambda_1)^{-1}=\dots=\lambda_r=\sigma(\lambda_r)^{-1}$.
It follows that $\Ker\alpha$ is the kernel $\WR^{(1)}\Gm$ of the norm map
$\WR\Gm\to\Gm$ sitting in $(\WR\Gm)^{\times r}$ diagonally.
The quotient $(\WR\Gm)^{\times r}/\Ker\alpha$ is therefore a torus isomorphic to
$\Gm\times(\WR\Gm)^{\times(r-1)}$ and to a closed subgroup $T\subset G$.
Since the dimension of the torus $T$ is $2r-1=n-1$ and $n-1$ is the rank of
$G$, the torus $T$ is maximal in $G$.

For $n=2r+1$, any $F$-point
$(\lambda,\lambda_1,\dots,\lambda_r)\in\WR^{(1)}\Gm(F)\times(\WR\Gm)^{\times r}(F)$
produces a preserving $h$ diagonal automorphism
$(\lambda,\lambda_1,\sigma(\lambda_1)^{-1},\dots,\lambda_r,\sigma(\lambda_r)^{-1})$
of $K^n$.
The kernel of the resulting homomorphism
$\WR^{(1)}\Gm\times(\WR\Gm)^{\times r}\to G$
is $\WR^{(1)}\Gm$, and the quotient by the kernel is the required torus
$(\WR\Gm)^{\times r}$ (of dimension $2r=n-1$).

The torus $T$ constructed in both cases is contained in the Borel subgroup
$B$ of $G$ which is the stabilizer of the flag of length $r$ of subspaces in $K^n$
where the $i$th subspace is the sum of the first $i$ odd summands of
$K^n$.
\end{proof}

\begin{rem}
\label{asmrem}
Actually, for {\em any} quasi-split semisimple adjoint affine algebraic group $G$
and for {\em any} Borel subgroup $B$ of $G$ (which exists because $G$ is quasi-split), {\em any}
maximal torus $T$ of $B$ (is maximal in $G$ and) is isomorphic to
$\WR_{L/F}\Gm$,
where $L$ is an \'etale $F$-algebra with the $\Gamma$-set $\Hom_{F\mathrm{\text{-}alg}}(L,\Fsep)$ isomorphic to
the set of vertices of the Dynkin diagram of $G$.
Indeed, this set of vertices is the set $X\subset (T_{\Fsep})^*$ of simple roots given by
$B$, and
we get that $T\simeq\WR_{L/F}\Gm$, because $X$ is a permutation basis of $(T_{\Fsep})^*$
($X$ generates $(T_{\Fsep})^*$ because $G$ is adjoint).

In particular, for $G$ as in Lemma \ref{T},
$T$ has the required isomorphism type because in this case
the action of $\Gamma$ on $X$ factors through the
Galois group of the quadratic separable field extension $K/F$ whose the
non-trivial element exchanges the vertices which are symmetric with respect to the
middle of the diagram.
It follows that $L\simeq K^{\times r}$ for $n=2r+1$ and $L\simeq F\times K^{\times(r-1)}$ for
$n=2r$ (note that $\#X=n-1$).
\end{rem}

For the rest of this section, we fix a torus $T\subset G$
and a Borel subgroup $B\subset G$ containing $T$ as in Lemma \ref{T}.


Let $E$ be a generic principle homogeneous space of $G$ as defined in \cite[\S3]{MR1999385}.
Thus $E$ is the generic fiber of certain $G$-torsor over certain smooth
geometrically irreducible $F$-variety having the versal property of
\cite[Definition 5.1(2)]{MR1999384} (we only need the weak version
\cite[Remark 5.8]{MR1999384} of the versality).

In particular, $E$ is a $G$-torsor over a regular field extension of $F$, not over $F$ itself.
For the sake of simplicity of notation, we let now $F$ be the field of
definition of $E$ (and $K$ be the quadratic extension of the new $F$ given by the tensor product
of the old $K$ and the new $F$ over the old $F$).

\begin{prop}
\label{E/B}
For $E$ as right above,
the image of $\CH(E/B)$ in the quotient
$$\CH(G/B)/\Im\big(\CH(G_K/B_K)\to\CH(G/B)\big)$$
is generated (as a ring) by elements of codimension $1$ and $2$.
\end{prop}

\begin{proof}
By Lemma \ref{B/T},
the image of
$$
\CH(E/B)\to\CH(G/B)/\Im\big(\CH(G_K/B_K)\to\CH(G/B)\big)
$$
is identified with the image of
$$
\CH(E/T)\to\CH(G/T)/\Im\big(\CH(G_K/T_K)\to\CH(G/T)\big).
$$
By \cite[Theorem 6.4]{MR2258262},
the image of $\CH(E/T)$ in $\CH(G/T)$ is identified with the image of
$\CH_T(\Spec F)$ in $\CH_T(G)=\CH(G/T)$.
The image
of
$$
\CH_T(\Spec F)\to\CH(G/T)/\Im\big(\CH(G_K/T_K)\to\CH(G/T)\big)
$$
coincides with the image of the quotient
$$
\CH_T(\Spec F)/\Im\big(\CH_T(\Spec K)\to\CH_T(\Spec F)\big).
$$
This quotient
is generated by its elements of codimension $1$ and $2$ by Corollaries
\ref{cor1} and \ref{cor2}.
\end{proof}

Note that for any $G$-torsor $E$ over $F$,
the kernel of the homomorphism $\CH(E/B)\to\CH(G/B)$ is the torsion
subgroup so that its image is identified with the Chow group $\CH(E/B)$
modulo torsion which we denote by $\BCH(E/B)$
(and call the {\em reduced} Chow group).
Furthermore, if $E$ is given by a hermitian form $h$,
then it splits over $K$;
therefore the image of the homomorphism
of $\CH(E/B)$ to $\CH(G/B)$ modulo the norms
is identified with the ring $\BCH(E/B)/N$, where
$N:=\Im\big(\CH(E_K/B_K)\to\BCH(E/B)\big)$ is the norm ideal.

We are going to consider a {\em generic hermitian form} $h$ of dimension $n$ (for a fixed
separable quadratic field extension $K/F$) by which we mean
the diagonal $K(t_1,\dots, t_n)/F(t_1,\dots,t_n)$-hermitian form
$\<t_1,\dots,t_n\>$, where $t_1,\dots,t_n$ are variables.
(For a justification of this definition, note that any hermitian form can be diagonalized,
\cite[Theorem 6.3 of Chapter 7]{MR770063}, so that an arbitrary $n$-dimensional $K/F$-hermitian form
is a specialization of $h$.)
Changing notation, we write now $F$ for $F(t_1,\dots,t_n)$ and $K$ for
$K(t_1,\dots,t_n)$.

\begin{cor}
\label{corE/B}
Let $Y$ be the variety of complete flags of totally isotropic subspaces of
the defined above generic hermitian form $h$.
Then the image of
$$
\CH(Y)\to\CH(G/B)/\Im\big(\CH(G_K/B_K)\to\CH(G/B)\big)
$$
is generated (as a ring) by the elements of
codimension $1$ and $2$.
Equivalently, the ring $\BCH(Y)/N$ is generated by codimension $1$ and
$2$.
\end{cor}

\begin{proof}
By the specialization argument as in Lemma \ref{spec},
for any regular field extension $F'/F$ and any $n$-dimensional $K'/F'$-hermitian form $h'$,
where $K':=K\otimes_FF'$,
the image of $\CH(Y)\to\CH(G/B)$ is contained in the image of
$\CH(Y')\to\CH(G/B)$, where $Y'$ is the variety of complete flags of $h'$.
Similarly,
the image of $\CH(E/B)\to\CH(G/B)$ is contained in the image of
$\CH(Y')\to\CH(G/B)$ for the $G$-torsor $E$ constructed below.
This $G$-torsor $E$ is given by certain
$n$-dimensional $K'/F'$-hermitian form $h'$.
It follows that
the image of $\CH(Y)\to\CH(G/B)$ coincides with the image of $\CH(E/B)\to\CH(G/B)$
so that the graded rings
$\BCH(Y)/N$ and $\BCH(E/B)/N$ are isomorphic and it
suffices to show that the latter ring is generated by elements of
codimension $\leq2$.

To construct the $G$-torsor $E$,
let us take a generic principle $G$-homogeneous space $E'$ of
\cite[\S3]{MR1999385} and climb over the
function field of the Weil transfer of the Severi-Brauer variety of the corresponding central
simple algebra.
In more details, the torsor $E'$ (changing notation, we may assume that it is defined over $F$)
corresponds to a degree $n$ central simple $K$-algebra $A$ endowed
with an $F$-linear unitary involution $\tau$, \cite[\S29D]{MR1632779}.
Over the function field $F(X)$ of the Weil transfer $X$ with respect to $K/F$ of the Severi-Brauer
$K$-variety of $A$,
the algebra $A$ splits and the involution $\tau$ becomes adjoint with respect
to some $K(X)/F(X)$-hermitian form.
We set $E:=E'_{F(X)}$.

Writing now $Y$ for the $F$-variety of complete flags of right $\tau$-isotropic ideals
in $A$, our aim is to show that $\CH(Y_{F(X)})$ modulo torsion and norms
is generated by codimensions $\leq2$.
Since the variety $Y$ is isomorphic to $E'/B$, we know by Proposition
\ref{E/B} that $\CH(Y)$ is so.
The projection $Y\times X\to Y$ is a Weil transfer of a projective bundle like the one considered in
\S\ref{Weil transfer of projective bundles}.
The $\CH(Y)$-algebra $\CH(Y\times X)$ is generated by codimension
$2$ by Proposition \ref{B}.
Since the pull-back $\CH(Y\times X)\to\CH(Y_{F(X)})$ is surjective,
the $\CH(Y)$-algebra $\CH(Y_{F(X)})$ is generated by codimension
$2$.
\end{proof}

\section
{Weil transfer of projective bundles}
\label{Weil transfer of projective bundles}

Let $K/F$ be a separable quadratic field extension.
Let $X$ be a smooth geometrically irreducible $F$-variety.
Let $\cE$ be a vector bundle over $X_K$.
We are going to consider the Weil transfer $Y\to X$ of the projective
bundle $P\to X_K$ of $\cE$ with respect to the base change morphism $X_K\to X$.
We refer to \cite[\S7.6]{MR1045822} and \cite[\S4]{MR1321819} for generalities on Weil
transfers of schemes (also called {\em Weil restriction} and {\em corestriction} in the literature).

We will need the following statement, where $P$ can be any variety with a
morphism to $X_K$:

\begin{lemma}
\label{tak Weil}
The Weil transfer of $P$ with respect to $X_K\to X$
can be obtained as the Weil transfer with respect to $K/F$ (which produces
an $\WR_{K/F}(X_K)$-scheme) followed by the base change with respect to the
``diagonal'' morphism $X\to\WR_{K/F}(X_K)$ corresponding to the identity under
the identification
$\Mor_F(X,\WR_{K/F}(X_K))=\Mor_K(X_K,X_K)$.
\end{lemma}

\begin{proof}
We write $\WR$ for $\WR_{K/F}$.
Let $Y$ be the fibred product $\WR P\times_{\WR X_K} X$.
To show that $Y$ is the Weil transfer of $P$ with respect to $X_K\to X$,
it suffices to check that for any $X$-scheme $S$,
the set
$\Mor_X(S,Y)$ is naturally identified with the set $\Mor_{X_K}(S_K,P)$
(note that $S_K:=S\times_F\Spec K=S\times_XX_K$).
By properties of the fibred product, we have a natural identification
\begin{equation}
\label{raz}
\Mor_X(S,Y)=\Mor_{\WR X_K}(S,\WR P).
\end{equation}
The set on the right hand side of this equality is a subset of the set $\Mor_F(S,\WR P)$
which is naturally identified with the set $\Mor_K(S_K,P)$.
The subset
$$
\Mor_{\WR X_K}(S,\WR P)\subset \Mor_F(S,\WR P)
$$
corresponds under this identification to the
subset $\Mor_{X_K}(S_K,P)\subset\Mor_K(S_K,P)$:
\begin{equation}
\label{dva}
\Mor_{\WR X_K}(S,\WR P)=\Mor_{X_K}(S_K,P).
\end{equation}
Indeed, by naturalness of the identification $\Mor_F(S,\WR P)=\Mor_K(S_K,P)$
in the right argument, the morphism $S_K\to X_K$ corresponding to the
composition of a given morphism $S\to\WR P$ with the morphism $\WR P\to\WR X_K$,
is the composition $S_K\to P\to X_K$.
Also note that the morphism $S_K\to X_K$ corresponding to the morphism
$$
(S\to\WR X_K)=(S\to X\to\WR X_K)
$$
is obtained from $S\to X$ by the base change (to see it use the
naturalness of the identification $\Mor(S,\WR X_K)=\Mor(S_K,X_K)$ in the left
argument).

Composing identifications (\ref{raz}) and (\ref{dva}), we get the identification required.
\end{proof}


For any integer $i$,
we have a map $\CH^i(P)\to\CH^{2i}(Y)$ (just a map, not a homomorphism) defined as the
composition of the map $\CH^i(P)\to\CH^{2i}(\WR_{K/F} P)$ of \cite[\S3]{MR1809664} followed by the
pull-back homomorphism $\CH^{2i}(\WR_{K/F} P)\to\CH^{2i}(Y)$
(we have in mind the construction of $Y$ given in Lemma \ref{tak Weil} here).
The first Chern class of the tautological line vector bundle on $P$
(the tautological vector bundle on $P$ corresponds to the locally free $\cO_P$-module $\cO_P(-1)$)
gives us therefore an element $c\in\CH^2(Y)$.

\begin{prop}
\label{B}
The $\CH(X)$-algebra $\CH(Y)/\Im\big(\CH(Y_K)\to\CH(Y)\big)$ is generated by the class of the element
$c$.
\end{prop}

\begin{proof}
We start by the case where $X=\Spec F$.
In this case, the ring $\CH(Y)$ is identified with the subring $R^\sigma$
of the
$\sigma$-invariant elements in the ring $R:=\CH(\PS^{r-1}\times\PS^{r-1})$,
where $r$ is the rank of the vector bundle $\cE$ and $\sigma$ is the
factor exchange involution on $R$.
The image of the norm map is the image of the homomorphism
$(1+\sigma):R\to R^\sigma$.
The ring $R^\sigma$ is generated by two elements: $h\times h$ and $1\times
h+h\times1$.
Therefore the quotient $R^\sigma/(1+\sigma)R$ is generated the class of $h\times
h$ which is the class corresponding to the element $c$ of the statement.

In order to do the general case,
we apply a variant of \cite[Statement 2.13]{Vishik-u-invariant}.
Let for a moment $X$ and $Y$ be arbitrary smooth $F$-varieties with
irreducible $X$.
For a morphism $\pi:Y\to X$,
consider the (finite) filtration
$$
\CH(Y)=\cF^0\CH(Y)\supset\cF^1\CH(Y)\supset\dots
$$
on $\CH(Y)$ with $\cF^i\CH(Y)$ being the
subgroup generated by the classes of cycles on $Y$ whose image in $X$ has
codimension $\geq i$.
For any point $x\in X$, let $Y_x$ be the fiber of $\pi$ over $x$.
For any $i$, let $X^{(i)}$ be the set of points of $X$ of codimension $i$.
There is a surjection
\begin{equation}
\label{surjekcija}
\Oplus_{x\in X^{(i)}}\CH(Y_x)\to\cF^i\CH(Y)/\cF^{i+1}\CH(Y),
\end{equation}
mapping an element $\alpha\in\CH(Y_x)$ to the class modulo $\cF^{i+1}\CH(Y)$ of the image
under the push-forward $\CH(Y_T)\to\cF^i\CH(Y)$ of an arbitrary preimage
of $\alpha$ under the pull-back epimorphism $\CH(Y_T)\onto\CH(Y_x)$, where
$T\subset X$ is the closure of $x$ and $Y_T:=Y\times_XT\hookrightarrow Y$ is
the preimage of $T$ under $Y\to X$.

\begin{lemma}[{cf. \cite[Statement 2.13]{Vishik-u-invariant}}]
\label{statement}
Let $\pi:Y\to X$ be as above and
let $\zeta$ be the generic point of $X$.
Let $B\subset\CH(Y)$ be a $\CH(X)$-submodule such that
\begin{enumerate}
\item[(a)]
the composition
$B\hookrightarrow\CH(Y)\to\CH(Y_\zeta)$ is surjective and
\item[(b)]
for any $x\in X$
either the specialization homomorphism $s_x:\CH(Y_\zeta)\to\CH(Y_x)$ is surjective
or the image of
$\CH(Y_x)\to\cF^i\CH(Y)/\cF^{i+1}\CH(Y)$ is
contained in the image of $B\cap\cF^{i}\CH(Y)\to\cF^i\CH(Y)/\cF^{i+1}\CH(Y)$.
\end{enumerate}
Then $B=\CH(Y)$.
\end{lemma}

\begin{proof}
We are repeating the proof of \cite[Statement 2.13]{Vishik-u-invariant}
making necessary modifications.
If the specialization homomorphism $s_x$ is surjective for a point $x\in
X^{(i)}$, then the image of $\CH(Y_x)$ under (\ref{surjekcija}) is
contained in the image of $[T]\cdot B\subset \cF^i\CH(Y)\cap B$ in the quotient $\cF^i\CH(Y)/\cF^{i+1}\CH(Y)$.
Otherwise, we know already by the hypothesis that it is contained in the image of $B\cap\cF^i\CH(Y)$.
\end{proof}

We are turning back to the proof of Proposition \ref{B} and our particular
$Y\to X$.
We apply Lemma \ref{statement},
taking for
$B$ the $\CH(X)$-submodule of $\CH(Y)$ generated by the powers of $c$
and the image of $\CH(Y_K)\to\CH(Y)$.
For any $x\in X$, the fiber $Y_x$
is isomorphic to the Weil transfer of $\PS^{r-1}$ with respect to the quadratic
\'etale $F(x)$-algebra $F(x)\otimes_FK$.
If $\zeta$ is the generic point, $F(\zeta)\otimes_FK$ is a field.
Therefore condition (a) of Lemma \ref{statement}, requiring that
the homomorphism $B\to\CH(Y_\zeta)$
is surjective, is satisfied.

Condition (b) of Lemma \ref{statement} is
satisfied as well.
Indeed, the specialization homomorphism $\CH(Y_\zeta)\to\CH(Y_x)$,
$x\in X$,
is surjective if the residue field of the point $x$ does not contain a subfield
isomorphic to $K$.
We finish the proof by showing that in the opposite case the image of
$\CH(Y_x)\to\cF^i\CH(Y)/\cF^{i+1}\CH(Y)$ (for $i$ such that $x\in X^{(i)}$) is
in the image of $\cF^i\CH(Y_K)\to\cF^i\CH(Y)/\cF^{i+1}\CH(Y)$.

Let $T$ be the closure of $x$ in $X$.
Let $Y_T=Y\times_XT\hookrightarrow Y$ be the preimage of $T$ under $Y\to
X$.
The image of the homomorphism $\CH(Y_x)\to \cF^i\CH(Y)/\cF^{i+1}\CH(Y)$
coincides with the image of the homomorphism
$\CH(Y_T)\to\cF^i\CH(Y)/\cF^{i+1}\CH(Y)$ induced by the push-forward.
Since $x$ is the generic point of $T$ and $F(x)=F(T)\supset K$, a
non-empty open subset $U\subset T$ possesses a morphism to $\Spec K$.
Its preimage $Y_U\subset Y_T$ is open and also possesses a morphism to $\Spec
K$.
Therefore $(Y_U)_K\simeq Y_U\coprod Y_U$ (as $F$-varieties) and, in particular,
the push-forward $\CH(Y_U)_K\to\CH(Y_U)$ is surjective.

The varieties and morphisms in play fit in the following commutative diagram:
$$
\divide\dgARROWLENGTH by9
\begin{diagram}
\node{Y_K}\arrow[2]{e}\node{}\node{Y}\arrow[2]{e}\node{}\node{X}\\
\node{(Y_T)_K}\arrow{n,J}\arrow[2]{e}\node{}\node{Y_T}\arrow[2]{e}\arrow{n,J}\node{}
\node{T}\arrow{n,J}\\
\node{(Y_U)_K}\arrow{n,J}\arrow[2]{e}\node{}\node{Y_U}\arrow{n,J}\arrow[2]{e}\node{}\node{U}
\arrow{n,J}
\end{diagram}
$$

It follows that the image of the push-forward
$\CH(Y_T)_K\to\CH(Y_T)$ generates $\CH(Y_T)$ modulo the image of
$\CH(Y_T\setminus Y_U)$.
Since the image of $\CH(Y_T\setminus Y_U)\to\CH(Y)$ is in
$\cF^{i+1}\CH(Y)$, it follows that
the image of $\CH(Y_T)\to\cF^i\CH(Y)/\cF^{i+1}\CH(Y)$
is contained in the image of $\cF^i\CH(Y_K)\to\cF^i\CH(Y)/\cF^{i+1}\CH(Y)$.
%
\end{proof}

We write $\BCH(Y)/N$ for the reduced Chow ring $\BCH(Y)$ modulo the norm
ideal $N:=\Im\big(\BCH(Y_K)\to\BCH(Y)\big)$.
In particular, $\BCH(X)/N$ is defined as
$\BCH(X)/\Im\big(\BCH(X_K)\to\BCH(X)\big)$.

As a particular case of the map $\CH^i(P)\to\CH^{2i}(Y)$ considered right before Proposition \ref{B},
we have a map $\CH^i(X_K)\to\CH^{2i}(X)$
and we write $c_i\in\CH^{2i}(X)$ for the class of the image of the Chern class
$c_i(\cE)\in\CH^i(X_K)$.

\begin{cor}
\label{corWeil}
The $\BCH(X)/N$-algebra $\BCH(Y)/N$
is generated by the class of $c$ subject
to the only one relation $\Sum_{i=0}^rc_ic^{r-i}=0$, where $r$ is the rank
of the vector bundle $\cE$.
\end{cor}

\begin{proof}
Let $\sigma$ be the nontrivial automorphism of $K/F$.
The variety $Y_K$ is identified with the product $P\times_{X_K}P'$, where
$P'\to X_K$ is the base change of $P\to X_K$ by the automorphism of $X_K$ induced by
$\sigma$;
clearly, $P'\to X_K$ is the projective bundle of the vector bundle $\cE'$
obtained from $\cE$ by the base changed induced by $\sigma$.
It follows that the $\CH(X_K)$-algebra $\CH(Y_K)$ is generated by two
elements $a:=c_1(\cO_P(-1))$ and $b:=c_1(\cO_{P'}(-1))$ subject to two
relations:
\begin{equation}
\label{two relations}
\Sum_{i=0}^r(-1)^ic_i(\cE)\cdot a^{r-i}=0
\;\;\text{ and }\;\;
\Sum_{i=0}^r(-1)^ic_i(\cE')\cdot b^{r-i}=0.
\end{equation}

The action of $\sigma$ on $\CH(X_K)$ exchanges the Chern classes of $\cE$
with the Chern classes of $\cE'$.
The action of $\sigma$ on $\CH(Y_K)$ exchanges $a$ with $b$.
Therefore the product of relations (\ref{two relations}) considered in the group
$\CH(Y_K)$ modulo the image of the composition $\CH(Y_K)\to\CH(Y)\to\CH(Y_K)$
gives the relation
$$
\Sum_{i=0}^rc_i(\cE)c_i(\cE')\cdot (ab)^{r-i}=0.
$$
Note that $ab\in\CH(Y_K)$ is the image of $c\in\CH(Y)$ and $c_i(\cE)c_i(\cE')\in\CH(X_K)$ is the image of
$c_i\in\CH(X)$.
Therefore,
the relation between the powers of $c$ of Corollary \ref{corWeil} holds in $\BCH(Y)/N$
because it holds in $\BCH(Y_K)/N\supset \BCH(Y)/N$.
(Here, abusing notation, we write $\CH(Y_K)/N$ for the quotient of
$\BCH(Y_K)$ by its subgroup $N:=\Im\big(\BCH(Y_K)\to\BCH(Y)\big)$.)

In particular, the $\BCH(X)/N$-module $\BCH(Y)/N$ is generated by
$1,c,\dots,c^{r-1}$.
The module generators are free, because they are free in $\BCH(Y_K)^\sigma/N\supset\BCH(Y)/N$
(see below).
It follows that the relation is the only one.

To finish the proof, let us consider the subgroup $\BCH(Y_K)^\sigma$ of
the $\sigma$-invariant elements in $\BCH(Y_K)$.
This subgroup contains $N$ and
the quotient $\BCH(Y_K)^\sigma/N$ is a $\BCH(X_K)^\sigma/N$-algebra
generated by $c$ subject to the one relation.
In particular, the elements $1,c,\dots,c^{r-1}$ are free over
$\BCH(X_K)^\sigma/N$ and therefore also free over $\BCH(X)/N\subset\BCH(X_K)^\sigma/N$.
\end{proof}

\begin{rem}
\label{c2i}
The element $c_i\in\CH^{2i}(X)$, defined right before Corollary \ref{corWeil}, can be also defined as the Chern class
$c_{2i}(\cE\to X)$, where $\cE\to X$ is the vector bundle
obtained by composition of the vector bundle $\cE\to X_K$ and the trivial vector bundle $X_K\to X$
(to see the structure of the vector bundle on $X_K\to X$,
note that the morphism $X_K\to X$ is a base change of the rank $2$ trivial vector bundle $\Spec K\to\Spec
F$).
Although $c_{2i}(\cE\to X)$ is different from $c_i$, their classes in $\BCH(X)/N$ coincide,
because the image of $c_{2i}(\cE\to X)$ in $\BCH(X_K)$ (which is easy to compute because the vector bundle $\cE\to X$
over $X_K$ is isomorphic to $\cE\oplus\cE'$) is congruent modulo $N$ to $c_i$.

Similarly, the element $c$ can be defined as the 2nd Chern class of the
rank $2$ vector bundle on $Y$ obtained as the Weil transfer of the rank
$1$ tautological vector bundle on $P$.
\end{rem}

\section
{Generic maximal unitary grassmannians}
\label{Generic maximal grassmannians}

Let now $h$ be a {\em generic} $K/F$-hermitian form of dimension
$n\geq2$ (defined right before Corollary \ref{corE/B}).
Let $r:=\lfloor n/2\rfloor$.
We consider the variety $X$ of $r$-dimensional totally isotropic subspaces
in $h$.
We keep the notation $G$ of the second half of Section \ref{Generic varieties of complete
flags}:
$G$ is a non-split quasi-split adjoint absolutely simple affine algebraic group
over a field $F$ of type $\cA_{n-1}$ becoming split over $K$.
Let $P$ be a parabolic subgroup
in $G$ representing the conjugacy class corresponding to $X$.


The aim of this section is to describe the image of the homomorphism
$$
\CH(X)\to\CH(G/P)/\Im\big(\CH(G_K/P_K)\to\CH(G/P)\big)
$$
(constructed (in a more general situation) right before Lemma \ref{spec}).
We recall that the $G$-torsor given by $h$ splits over $K$;
therefore the image of the above homomorphism is identified with the ring
$\BCH(X)/N$, where $N:=\Im\big(\BCH(X_K)\to\BCH(X)\big)$.
(So, $\BCH(X)/N$ is the ring
$\CH(X)$
modulo the sum of the ideal of the torsion elements and the ideal of the
norms.)

\begin{prop}
\label{1.9}
The components of positive codimension of the ring $\BCH(X)/N$ are
trivial.
\end{prop}

\begin{proof}
Let $I=[1,\;r]=\{1,2,\dots,r\}$.
For every $J\subset I$ we consider the variety $X_J$ of flags of totally
isotropic subspaces in $h$ of dimensions given by $J$.

By induction on $l\in I$, we prove the following statement:
the ring $\BCH(X_{[l,\;r]})/N$ (where $N$ here is the image of norm homomorphism for $X_{[l,\;r]}$)
is generated by its elements of codimensions $1$,
its elements of codimension $2$, and the Chern classes of the tautological rank $2l$ vector bundle
$\cT_l$ on $X_{[l,\;r]}$.
Note that this statement for $l=r$ gives the statement of Proposition \ref{1.9}
because $X_{\{r\}}=X$ and we have the following triviality statements for the generators:
according to \cite[Proposition 3.9]{isouni} the elements of codimension $2$ as well as the Chern
classes of positive codimension of the tautological vector bundle on $G/P$ are trivial in
$\CH(G/P)/N$;
for odd $n$, again by \cite[Proposition 3.9]{isouni}, the elements of codimension $1$
are also trivial in $\CH(G/P)/N$;
finally, if $n$ is even, the elements of codimension $1$ in $\CH(X)$
become
trivial in $\CH(G/P)/N$ by Lemma \ref{6.2} below (because the discriminant of our hermitian form $h$ is
non-trivial).

The induction base $l=1$
is given by Corollary \ref{corE/B}.
Now, assuming that $l\geq2$, let us do the passage from $l-1$ to $l$.

The projection $X_{[l-1,\;r]}\to X_{[l,\;r]}$ is
the Weil transfer of a projective bundle
(namely, of the projective bundle over $(X_{[l,\;r]})_K$
given by the dual of (any)one of two rank $l$ tautological vector bundles
on $(X_{[l,\;r]})_K$).
Therefore, by Corollary \ref{corWeil}, the
$\BCH(X_{[l,\;r]})/N$-algebra $\BCH(X_{[l-1,\;r]})/N$ is generated by
certain
codimension $2$ element $c$ subject to one relation
$\Sum_{i=0}^lc_ic^{l-i}=0$, where $c_i:=c_{2i}(\cT_l)$
(see Remark \ref{c2i}).
In particular, the
$\BCH(X_{[l,\;r]})/N$-module $\BCH(X_{[l-1,\;r]})/N$ is free of rank $l$.

Now let $C\subset\BCH(X_{[l,\;r]})/N$ be the subring
generated by all $c_i$ together with the elements of codimensions $1,2$.
The coefficients of the above relation are then in $C$.
Therefore the subring of $\BCH(X_{[l-1,\;r]})/N$ generated
by $C$ and $c$ is also free (now as a $C$-module) of rank $l$.
On the other hand, this subring coincides with the total ring by the
induction hypothesis.
Indeed, it contains all the elements of codimension $1,2$ in
$\BCH(X_{[l-1,\;r]})/N$ because any such element is a polynomial in $c$
with coefficients of codimension $\leq2$ in $\BCH(X_{[l,\;r]})/N$.
It also contains the Chern classes of the tautological rank $2(l-1)$
vector bundle $\cT_{l-1}$ on $X_{[l-1,\;r]}$ because $c=c_2(\cT_l/\cT_{l-1})$ so that
the Chern classes of $\cT_{l-1}$ are expressible in terms of
$c_i$ and $c$
(note that the odd Chern classes of $\cT_l$ and of $\cT_l/\cT_{l-1}$ are trivial in
$\BCH(X_{[l-1,\;r]\;K})/N\supset\BCH(X_{[l-1,\;r]})/N$).

It follows that $C=\BCH(X_{[l,\;r]})/N$.
\end{proof}

We terminate this section by a study of the case of $n=2r$.
We do not assume anymore that our $n$-dimensional hermitian form $h$ is
generic, it is arbitrary.
We still write $X$ for the variety of $r$-dimensional totally isotropic subspaces
in $h$.
The variety $X_K$ is identified with the $K$-grassmannian of $r$-planes in
$V$ (the vector space of definition of $h$).
Let $\cE$ be the tautological bundle on the grassmannian.
For any $i\geq 0$, we set $c_i:=c_i(\cE)\in\CH^i(X_K)$.
The ring $\CH(X_K)$ is generated by the elements $c_i$.
In particular, $\CH^1(X_K)$ is (an infinite cyclic group) generated by $c_1$.
We recall that $\BCH(X)$ stands for the reduced Chow group and coincides with the image of $\CH(X)\to\CH(X_K)$.

The {\em discriminant} $\disc h$ of the hermitian form $h$
is the class in $F^\times/N(K^\times)$ of the signed determinant
$$
(-1)^{n(n-1)/2}\cdot\det(h(e_i,e_j))_{i,j\in[1,\;n]}
$$
of the Gram matrix of $h$ in a basis $e_1,\dots,e_n$ of $V$.

\begin{lemma}
\label{6.2}
The group $\BCH^1(X)$ is generated by
\begin{itemize}
\item[]
$c_1$, if $\disc h$ is trivial;
\item[]
$2c_1$, if $\disc h$ is non-trivial.
\end{itemize}
\end{lemma}

\begin{proof}
The cokernel of $\CH^1(X)\to\CH^1(X_K)$ is isomorphic to the kernel of
$\Br(F)\to\Br(F(X))$.
This kernel is generated by the class of the {\em
discriminant algebra} of $h$, \cite{MR1628279}.
The discriminant algebra of $h$ is the quaternion algebra given by the
quadratic extension $K/F$, the non-trivial automorphism of $K/F$ and
$\disc h$.
Its Brauer class is always killed by $2$ and is non-trivial if and only if
$\disc h$ is non-trivial.
\end{proof}

\section
{Essential motives of unitary grassmannians}
\label{Isotropic grassmannians}

We fix the following notation.
Let $K/F$ be a separable quadratic field extension.
Let $n$ be an integer $\geq0$.
Let $V$ be a vector space over $K$ of dimension $n$.
Let $h$ be a $K/F$-hermitian form on $V$.
For any integer $r$,
let $X_r$ be the $F$-variety
of $r$-dimensional totally isotropic subspaces in $V$
($X_r$ is a closed subvariety of the Weil transfer with respect to $K/F$ of
the $K$-grassmannian of $r$-planes in $V$; $X_0=\Spec F$; $X_r=\emptyset$ for $r$ outside of the interval
$[0,\;n/2]$).

We are working with the Grothendieck Chow motives with coefficients in
$\F_2$, \cite[Chapter XII]{EKM}.
We write $M(X)$ for the motive of a smooth projective $F$-variety $X$.
We are constantly (and without any further reference) using the
Krull-Schmidt principle for the motives of quasi-homogeneous varieties,
\cite[Corollary 2.2]{outer}.

\begin{lemma}[{\cite[Theorem 15.8]{MR1758562}}]
\label{iso-whole}
Assume that the hermitian form $h$ is isotropic:
$n$ is $\geq2$ and $h\simeq\HH\bot h'$, where
$\HH$ is the hyperbolic plane, $h'$ a hermitian form of dimension $n-2$.
For any integer $r$
one has
$$
M(X_r)\simeq M(X'_{r-1})\oplus M(X'_r)(i)\oplus M(X'_{r-1})(j)\oplus M,
$$
where $X'_{r-1}$ and $X'_r$ are the varieties of $h'$,
$i=(\dim X_r-\dim X'_r)/2$, $j=\dim X_r-\dim X'_{r-1}$,
and $M$ is a sum of shifts of the $F$-motive of $\Spec K$.
\end{lemma}

The following Corollary is also a consequence of a general result of
\cite{outer} or of \cite{MR2110630}:

\begin{cor}
If $h$ is {\em split} (meaning {\em hyperbolic} for even $n$ or {\em ``almost hyperbolic''} for odd $n$),
then $M(X_r)$ is a sum of shifts of the motives of $\Spec
F$ and of $\Spec K$.
\qed
\end{cor}

\begin{cor}
There is a decomposition $M(X_r)\simeq M_r\oplus M$ such that the motive $M$ is a sum
of shifts of $M(\Spec K)$ and for any field extension $L/F$ with split
$h_L$ the motive $M_r$ is {\em split} (meaning {\em is a sum of Tate motives}).
\end{cor}

\begin{proof}
Apply \cite[Proposition 4.1]{unitary} inductively to $E:=F(X_{\lfloor n/2\rfloor})$ and $S:=\Spec
K$.
Note that the variety $S_E$ is still irreducible and has indecomposable
motive because $K\otimes_FE$ is a field.
The hermitian form $h_E$ is split.
Since $X_{\lfloor n/2\rfloor}(K)\ne\emptyset$,
the sum $M$ of all copies of shifts of $M(\Spec K)$ present in the complete decomposition
of $M(X_r)$ over $E$, can be extracted from $M(X_r)$ over $F$.
The remaining part $M_r$ of the motive of $X_r$ has the desired
property.
\end{proof}

\begin{rem}
\label{rem7.4}
The reduced Chow group (homological or cohomological one) of the motive $M$
(as a subgroup of $\BCh(X_r)$)
is the
image $N$ of the norm map $\BCh((X_r)_K)\to\BCh(X_r)$
(it is evidently contained in $N$ and coincides in fact with $N$ because $N$ intersects the reduced Chow group of $M_r$ trivially).
Therefore the reduced Chow group of $M_r$ is identified with the quotient
$\BCh(X_r)/N$.
Here $\Ch(X_r)$ stands for the Chow group with coefficients in $\F_2$ of
the variety $X_r$ and the reduced Chow group $\BCh(X_r)$ is $\BCH(X_r)$ modulo $2$ or, equivalently,
the image of $\Ch(X_r)\to\Ch(\bar{X}_r)$.
In the quasi-split case (the case where $h$ is hyperbolic or almost
hyperbolic), the reduced Chow group in the above statements can be
replaced by the usual Chow group (still modulo $2$).
We refer to \cite[\S64]{EKM} for the definition of homological and
cohomological Chow group of a motive.
The coincidence of descriptions of homological and cohomological Chow groups for the
motives $M$ and $M_r$ is explained by their symmetry: $M\simeq M^*(\dim X_r)$
(and the same for $M_r$), where $M^*$ is the {\em dual} motive,
\cite[\S65]{EKM}.
\end{rem}

\begin{dfn}
\label{def-essential}
The motive $M_r$ (defined by $X_r$ uniquely up to an isomorphism) will be called the {\em
essential motive} of $X_r$ (or the {\em essential part} of the motive of $X_r$).
\end{dfn}

It follows that
the decomposition of the essential motive in the isotropic case has
precisely the same shape as the decomposition of the motive of an
isotropic orthogonal grassmannian \cite{gog}:

\begin{cor}
\label{iso-essential}
Under the hypotheses of Lemma \ref{iso-whole},
one has
$$
M_r\simeq M'_{r-1}\oplus M'_r(i)\oplus M'_{r-1}(j),
$$
where $M'_{r-1}$ and $M'_r$ are the essential motives of $X'_{r-1}$ and $X'_r$,
$i=(\dim X_r-\dim X'_r)/2$, $j=\dim X_r-\dim X'_{r-1}$.
\qed
\end{cor}

According to the general result of \cite{outer}, any summand of the
complete motivic decomposition of the variety $X_r$ is a shift of the
upper motive $U(X_s)$ for some $s\geq r$ or a shift of the motive of the $F$-variety
$\Spec K$.
Therefore we get

\begin{cor}
\label{any is U}
Any summand of the complete decomposition of the essential motive $M_r$ is
a shift of the
upper motive $U(X_s)$ for some $s\geq r$.
\qed
\end{cor}

\begin{rem}
\label{dimMr}
A motive is {\em split} if it is isomorphic to a finite direct sum of Tate
motives.
A motive is {\em geometrically split} if it becomes split over an
extension of the base field.
{\em Dimension} $\dim P$ of a geometrically split motive $P$ is the
maximum of $|i-j|$ for $i$ and $j$ running over the integers such that the Tate motives $M(\Spec
L)(i)$ and $M(\Spec L)(j)$ are direct summands of $P_L$, where $L/F$ is a
field extension splitting $N$.

Since the field $F$ is algebraically closed in the function field $L:=F(X_{\lfloor
n/2\rfloor})$ and the motive of $(X_r)_L$ contains the Tate summands
$\F_2$ and $\F_2(\dim X_r)$, these Tate motives are summands of $(M_r)_L$.
It follows that $\dim M_r=\dim X_r$.
\end{rem}



\section
{Generic unitary grassmannians}
\label{Generic grassmannians}

In the statements below we use the notion of the {\em essential motive}
$M_r$ of the variety $X_r$, introduced in the previous section.
It turns out that in the generic case, this motive is indecomposable:

\begin{thm}
\label{main}
Let $h$ be a {\em generic} $K/F$-hermitian form of an arbitrary dimension
$n\geq0$.
For $r=0,1,\dots,\lfloor n/2\rfloor$,
the essential motive $M_r$ of the variety $X_r$ is indecomposable,
the variety
$X_r$ is $2$-incompressible.
\end{thm}

\begin{proof}
We induct on $n$ in the proof of the first statement.
The induction base is the trivial case of $n<2$.
Now we assume that $n\geq2$.

We do a descending induction on $r$.
The case of the maximal $r=\lfloor n/2\lfloor$ is an immediate consequence of Proposition
\ref{1.9}.
Indeed, one summand of the complete decomposition of the motive $M_r$ for such $r$
is the upper motive $U(X_r)$ of $X_r$.
The remaining summands (if any) are positive shifts $U(X_r)(i)$ ($i>0$) of the upper
motive, see Corollary \ref{any is U}.
But if we have a summand $U(X_r)(i)$, then the reduced Chow group
$\BCh^i(M_r)$ is non-zero.
However by Remark \ref{rem7.4}, $\BCh(M_r)$ is isomorphic to $\BCH(X_r)/N$ which is $0$ in positive codimensions
by Proposition \ref{1.9}.

Now we assume that $r<\lfloor n/2\rfloor$.
Since the case of $r=0$ is trivial, we may assume that $r\geq1$
and $n\geq4$.

Let $L:=F(X_1)$.
We have $h_L\simeq\HH\bot h'$, where $h'$ is a $K(X_1)/F(X_1)$-hermitian form
of dimension $n-2$ and $\HH$ is the hyperbolic plane.

For any integer $s$, we write $X'_s$ for the variety $X_s$ of the hermitian form
$h'$, and we write
$M'_s$ for the
essential motive of the variety $X'_s$.
By Corollary \ref{iso-essential},
the motive $(M_r)_L$ decomposes in a sum of three summands:
\begin{equation}
\label{tri}
(M_r)_L\simeq M'_{r-1}\oplus M'_r(i)
\oplus M'_{r-1}(j),
\end{equation}
where $i:=(\dim X_r-\dim X'_r)/2$ and $j:=\dim X_r-\dim X'_{r-1}$.
Let us check that each of three summands of decomposition (\ref{tri}) is indecomposable.

Let $F'$ be the function field of the variety of $1$-dimensional totally
isotropic subspaces of the $K/F$-hermitian form $\<t_{n-1},t_n\>$.
Then $h_{F'}\simeq \HH\bot\<t_1,\dots,t_{n-2}\>$ so that we have a motivic
decomposition similar to (\ref{tri}) where each of the three summands is
indecomposable by the induction hypothesis.
Since the field extension $F'(X_1)/F'$ is purely transcendental, the
complete decomposition of $(M_r)_{F'(X_1)}$ has only three summands.
Since $L=F(X_1)\subset F'(X_1)$, the complete decomposition of $(M_r)_L$
has at most three summands so that the summands of
decomposition (\ref{tri}) are indecomposable.

It follows
(taking into account the duality like in \cite[Remark 1.3]{gog}, see also
\cite[Proposition 2.4]{sgog})
that if the motive $M_r$ is decomposable (over $F$), then
it has a summand $P$ with $P_L\simeq M'_r(i)=U(X'_r)(i)$.
Note that $U(X'_r)\simeq U((X_{r+1})_L)$.
By an analogue of \cite[Lemma 1.2]{gog} (see also
\cite[Proposition 2.4]{sgog}), $P\simeq U(X_{r+1})(i)$, showing that
$U(X_{r+1})_L\simeq M'_r$.
By the induction hypothesis, the motive $M_{r+1}$ is indecomposable, that
is, $U(X_{r+1})=M_{r+1}$.
Therefore we have an isomorphism $(M_{r+1})_L\simeq M'_r$ and,
in particular, $\dim X_{r+1}=\dim X'_r$ (see Remark \ref{dimMr}).
However
$\dim X_{r+1}= (r+1)\big(2n-3(r+1)\big)$,
$\dim X'_r=r\big(2(n-2)-3r\big)$, and the difference is
$2n-2r-3\geq n-3>0$ (recall that $n\geq4$ now).

To show that $X_r$ is $2$-incompressible, we show that its canonical
$2$-dimension $\cd_2 X_r$ equals $\dim X_r$.
By \cite[Theorem 5.1]{canondim}, $\cd_2 X_r=\dim U(X_r)$.
By the first part of Theorem \ref{main}, $U(X_r)=M_r$.
Finally, $\dim M_r=\dim X_r$ (see Remark \ref{dimMr}).
\end{proof}


\section
{Connection with quadratic forms}
\label{Applications to quadratic forms}

Let $K/F$ be a separable quadratic field extension extension, $V$ a
finite-dimensional vector space over $K$, $h$ a non-degenerate $K/F$-hermitian form on
$V$.
For any $v\in V$ the value $h(v,v)$ is in $F$ and
the map $q:V\to F$, $v\mapsto h(v,v)$ is a non-degenerate quadratic form on $V$ considered as a vector
space over $F$.
Note that the dimension of $q$ is even.
The Witt indices of $h$ and $q$ are related as follows:

\begin{lemma}
\label{i2i}
$i(q)=2i(h)$.
\end{lemma}

\begin{proof}
For any integer $r\geq0$, the inequality $i(h)\geq r$ implies $i(q)\geq2r$.
Indeed,
if $i(h)\geq r$, $V$ contains a totally $h$-isotropic $L$-subspace $W$ of
dimension $r$.
This $W$ is also totally $q$-isotropic and has dimension $2r$ over $F$.
Therefore $i(q)\geq2r$.

To finish we prove by induction on $r\geq0$ that $i(q)\geq2r-1$ implies $i(h)\geq r$.
This is trivial for $r=0$.
If $r>0$ and $i(q)\geq2r-1$, then $q$ is isotropic.
But any $q$-isotropic vector is also $h$ isotropic, therefore the $K$-vector
space $V$ decomposes in a direct sum of $h$-orthogonal subspaces $V=U\oplus V'$
such that $h|_U$ is a hyperbolic plane.
The subspaces $U$ and $V'$ are also $q$-orthogonal and $q|_U$ is
hyperbolic (of dimension $4$).
For $h':=h|_{V'}$ and $q':=q|_{V'}$
it follows that $i(h')=i(h)-1$ and $i(q')=i(q)-2$, and we are done by the
induction hypothesis applied to $h'$ (of course, $q'$ is the quadratic form given by $h'$).
\end{proof}

\begin{cor}[N. Jacobson, 1940]
\label{h i q}
If the quadratic forms $q_1$ and $q_2$ corresponding to $K/F$-hermitian forms
$h_1$ and $h_2$ are isomorphic, then $h_1$ and $h_2$ are also isomorphic.
\end{cor}

\begin{proof}
The orthogonal sum $q_1\bot-q_2$ is the quadratic form corresponding to
$h_1\bot-h_2$.
If $q_1$ and $q_2$ are isomorphic, $q_1\bot-q_2$ is hyperbolic, therefore
$h_1\bot-h_2$ is also hyperbolic implying that $h_1$ and $h_2$ are
isomorphic.
\end{proof}


For any integer $r$,
let $X_r$ be the variety of totally $h$-isotropic $r$-dimensional $K$-subspaces in $V$
and let $Y_r$ be the
variety of totally $q$-isotropic $r$-dimensional $F$-subspaces in $V$.
The variety $Y_n$, where $n:=\dim V$, is not connected and has two
isomorphic connected components;
changing notation, we let $Y_n$ be one of its connected component in
this case.

\begin{cor}
\label{X i Y}
For any $r$, the upper motives of the varieties $X_r$, $Y_{2r}$, and
$Y_{2r-1}$ are isomorphic.
In particular, these varieties have the same canonical $2$-dimension.
\end{cor}

\begin{proof}
By Lemma \ref{i2i}, each of the three varieties possesses a rational map
to each other.
Therefore the upper motives are isomorphic by \cite[Corollary 2.15]{upper}.
For the statement on canonical dimension see \cite[Theorem 5.1]{canondim}.
\end{proof}

Let us recall that according to the original definition \cite[Definition 5.11(2)]{MR2148072}
due to A. Vishik
of the $J$-invariant $J(q)$ of a non-degenerate quadratic form $q$ over $F$ of
even dimension $2n$, $J(q)$ is a certain subset of the set of integers
$\{0,1,\dots,n-1\}$.
Unfortunately, in \cite[\S88]{EKM}, the name {\em $J$-invariant} and the notation $J(q)$
stand for the complement of the above subset
(with the ``excuse'' that this choice
simplifies several formulas involving the $J$-invariant).
In the present paper we are using the original definition and notation.

Theorem \ref{main} with Corollary \ref{X i Y} make it possible
to compute
for any $n$ the smallest (in the sense of inclusion) value of
the $J$-invariant
of a non-degenerate quadratic form
given by tensor product an $n$-dimensional bilinear form by a fixed binary
quadratic form.
Only the case of even $n$ is of interest, because
the $J$-invariant is $\{1,2,\dots,n-1\}$ (everything but $0$) for odd $n$.

\begin{cor}
\label{J(generic)}
For any even $n\geq2$,
the smallest value of $J$-invariant discussed right above
is
$
\{0,2,4,\dots,n-2\}
$
(the set of even integers from $0$ till $n-2$).
\end{cor}

\begin{proof}
By \cite[Proposition 88.8]{EKM}, the $J$-invariant of any non-degenerate quadratic form given by
tensor product of an $n$-dimensional bilinear form by a binary quadratic form, is a subset of the
given set.
Let $K/F$ be the separable quadratic field extension given by the
discriminant of the fixed binary quadratic form.
We consider the generic $n$-dimensional $K/F$-hermitian form $h$ and
calculate the $J$-invariant of the associated quadratic form $q$.
We are using the above notation for the varieties associated to $h$ and to
$q$.

By \cite[Theorem 90.3]{EKM}, the canonical $2$-dimension of $Y_n$ is $\dim Y_n$ minus the sum
of the elements of the $J$-invariant.
On the other hand, by Corollary \ref{X i Y} and Theorem \ref{main}, the
canonical $2$-dimension of $Y_n$ is equal to the dimension of $X_{n/2}$ which
is
$$
\dim X_{n/2}=n^2/4=
n(n-1)/2-(0+2+\dots+(n-2)).
$$
Therefore the $J$-invariant is equal to the set indicated.
\end{proof}

Turning back to an arbitrary (not necessarily generic) hermitian form, we
have

\begin{prop}
\label{dvaMr}
For any $r>0$, the motive $M_r\oplus M_r(\dim Y_{2r-1}-\dim X_r)$ is isomorphic to a summand
of the motive of $Y_{2r-1}$.
If $n>2$, then the motive $M_r\oplus M_r(\dim Y_{2r}-\dim X_r)$ is isomorphic to
a summand of the motive of $Y_{2r}$.
(For $n=2$, the motive $M_r$ is isomorphic to
a summand of the motive of $Y_{2r}$.)
\end{prop}

\begin{proof}
We start by checking that the statement holds in the {\em generic} case.
In this case, by Theorem \ref{main}, the motive $M_r$ is equal to the upper motive
$U(X_r)$ which, by Corollary \ref{X i Y}, is isomorphic to
$U(Y_{2r-1})$ as well as to $U(Y_{2r})$, indecomposable summands of
$M(Y_{2r-1})$ and of $M(Y_{2r})$.
So, $M_r$ is an indecomposable summand of $M(Y_{2r-1})$ and $M(Y_{2r})$.
By duality \cite[\S65]{EKM}, $M_r(\dim Y_{2r-1}-\dim X_r)$ is also an indecomposable summand of
$M(Y_{2r-1})$ and
$M_r(\dim Y_{2r}-\dim X_r)$ is also an indecomposable summand of
$M(Y_{2r})$.
It remains to notice that
$$
\dim Y_{2r-1}-\dim X_r=(2r-1)(2n-3r+1)-r(2n-3r)>0
$$
for any $n$ and any $r\leq n/2$ and that
$$
\dim Y_{2r}-\dim X_r=r(4n-6r-1)-r(2n-3r)>0
$$
for any $n>2$ and any $r\leq n/2$.

In the {\em general} case, the given hermitian form $h$ is a specialization
of the generic one.
The varieties $X_r$, $Y_{2r-1}$, and $Y_{2r}$ are therefore
specializations of the corresponding varieties associated to the generic
form.
By \cite[Proposition 20.3 and Corollary 20.3]{Fulton}, the specialization
homomorphism on Chow groups respects multiplication and commutes with pull-backs and
push-forwards so that it also commutes with composition of
correspondences.
Therefore, specializing projectors we get projectors and specializing
isomorphisms of motives we get isomorphisms of motives.
To see that the specialization of the essential
motive is the essential motive, apply Lemma \ref{spec}.
\end{proof}

The following statement has been proved in \cite{MR2563697}
(the existence of such a decomposition (without the connection of $M_1$ to $X_1$) has been established
before \cite{MR2563697} in \cite[Theorem 5.1]{MR2066515}):

\begin{cor}
\label{kvadrika}
$M(Y_1)=M_1\oplus M_1(1)$.
\end{cor}

\begin{proof}
Since $\dim Y_1-\dim X_1=1$, we know by Proposition \ref{dvaMr} that
$M_1\oplus M_1(1)$ is a summand of $M(Y_1)$.
Comparing the {\em ranks} (the number of summands in the complete decomposition over an algebraic closure)
of the motives , we see that $M(Y_1)=M_1\oplus M_1(1)$.
\end{proof}

\begin{rem}
By Theorem \ref{main},
the decomposition of Corollary \ref{kvadrika} is complete in the generic
case (that is to say, the summands of the decomposition are indecomposable).
\end{rem}

Let us consider the natural closed imbedding $\inc:X_r\hookrightarrow
Y_{2r}$ which we have because a totally $h$-isotropic $K$-subspace is also
a $q$-isotropic $F$-subspace.
We will assume that the dimension $n$ of the hermitian form $h$ is even.
In this case the image $N$ of the norm homomorphism
$\BCh(Y_{2r})_K\to\BCh(Y_{2r})$ is $0$ (because the discriminant of $q$ is trivial and therefore $Y_{2r}$
is a projective homogeneous variety of {\em inner} type), so that
$\BCh(Y_{2r})/N=\BCh(Y_{2r})$.

The following observation is due to Maksim Zhykhovich:

\begin{lemma}
\label{9.8}
The pull-back $\inc^*:\BCh(Y_{2r})\to\BCh(X_r)/N$ is surjective,
the push-forward $\inc_*:\BCh(X_r)/N\to\BCh(Y_{2r})$ is injective.
\end{lemma}

\begin{rem}
In the quasi-split case, another proof of surjectivity of $\inc^*$ is given in
\cite[Lemma 4.1]{isouni}.
\end{rem}

\begin{proof}[Proof of Lemma \ref{9.8}]
We are working with motives in this proof.
One may use the Chow motives with coefficients in $\F_2$, but since we are interested in the reduced Chow groups in the end,
it is more appropriate to use
the $\BCh$-motives: they are obtained by replacing $\Ch$ by $\BCh$ in the
construction of the Chow motives with coefficients in $\F_2$.

Let us also write $\inc$ for the graph of the imbedding $\inc$.
Let $\pi$ be a projector on $X_r$ giving the motive $M_r$.
There exists a correspondence $\alpha$ such that
$\pi\compose\inc^t\compose\alpha=\pi$.
Indeed, such a correspondence exists in the generic case and can be
obtained in the general case by specialization like in the end of the proof of Proposition \ref{dvaMr}.
Since $\BCh(M_r)$ is identified with $\BCh(X_r)/N$, we have $\pi_*(x)\equiv
x\pmod{N}$ for any $x\in\BCh(X_r)$.
It follows that $x\equiv\pi_*(x)=\pi_*\inc^*\alpha_*(x)\equiv\inc^*\alpha_*(x)$
showing that $\inc^*:\BCh(Y_{2r})\to\BCh(X_r)/N$ is surjective.

Similarly,
there exists a correspondence $\beta$ such that
$\beta\compose\inc\compose\pi=\pi$.
It follows that $x\equiv\beta_*\inc_*(x)$ for any $x\in\BCh(X_r)$
showing that $\inc_*:\BCh(X_r)/N\to\BCh(Y_{2r})$ is injective.
\end{proof}

Since the composition $\inc_*\compose\inc^*$ is the multiplication by the class
$[X_r]\in\Ch(Y_{2r})$ and $\inc^*$ is a ring homomorphism,
we get

\begin{cor}
\label{isochow}
The quotient ring $\BCh(X_r)/N$ is isomorphic to
$\BCh(Y_{2r})/I$, where the ideal $I\subset\BCh(Y_{2r})$ is the annihilator of the element
$[X_r]\in\BCh(Y_{2r})$.
\qed
\end{cor}

\begin{example}
\label{primerchik}
Let us consider the case of $n=2r$.
The ring $\Ch(\bar{Y}_{2r})$ is generated by
certain elements $e_1,\dots,e_{2r-1}$ of codimensions $1,\dots,2r-1$ subject to the relations
$e_i^2=e_{2i}$, where $e_j:=0$ for $j>2r-1$, \cite[\S3]{MR2148072} (the notation for the generators we are using comes
from \cite[\S86]{EKM}).
If $h$ is generic, then, according to the famous \cite[Main Theorem 5.8]{MR2148072}
(also exposed in \cite[Theorem 87.7]{EKM}) and Corollary \ref{J(generic)}, the subring $\BCh(Y_{2r})\subset\Ch(\bar{Y}_{2r})$ is
generated by the even-codimensional elements $e_2,e_4,\dots,e_{2r-2}$.
The class $[X_r]\in\Ch(\bar{Y}_{2r})$ is {\em rational} (that means {\em belongs to $\BCh(Y_{2r})$}) and non-zero (by
injectivity of $\inc_*$) so that $[X_r]=e_2e_4\dots e_{2r-2}$ as this product is
the only non-zero rational element in codimension
$$
\codim_{Y_{2r}} X_r=
r(2r-1)-r^2=r(r-1)=
2+4+\dots+(2r-2).
$$
It follows that
for $n=2r$ and hyperbolic $h$,
the ring $\Ch(X_r)/N$ (we put the usual Chow group instead of the reduced one because they both coincide by the reason
that $h$ is hyperbolic) is generated by elements
$e_1,e_3,e_5,\dots,e_{2r-1}$ of codimensions $1,3,5,\dots,2r-1$ subject to
relations $e_i^2=0$ for any $i$.
\end{example}

\begin{example}
\label{9.11}
We briefly describe the situation with an {\em odd} $n$.
In this situation, the norm homomorphism $\BCh(Y_{2r})_K\to\BCh(Y_{2r})$ can
be non-zero.
By this reason, $\BCh(Y_{2r})$ in the statements of Lemma \ref{9.8} and Corollary
\ref{isochow} has to be replaced by the quotient $\BCh(Y_{2r})/N$.
In particular, the ring $\BCh(X_r)/N$ is naturally isomorphic to the
quotient of the ring $\BCh(Y_{2r})/N$ by the annihilator of
$[X_r]\in\BCh(Y_{2r})/N$.

Now we assume that $n=2r+1$.
The variety $\bar{Y}_{2r}$ is a rank $2r$ projective bundle over
$\bar{Y}_{2r+1}$
(note that $Y_{2r+1}$ is the {\em maximal}  and $Y_{2r}$ is the {\em previous to the maximal}
orthogonal grassmannian of $q$).
The ring $\Ch(\bar{Y}_{2r})$ is generated by elements
$e_i\in\Ch^i(\bar{Y}_{2r})$, $i=1,2,\dots,2r$ and
$e\in\Ch^1(\bar{Y}_{2r})$ subject to the relations
$e_i^2=e_{2i}$ and $e^{2r+1}=0$.
The generators $e_i\in\Ch(\bar{Y}_{2r})$ here are pull-backs of the generators
$e_i\in\Ch(\bar{Y}_{2r+1})$ of the Chow group of the maximal orthogonal
grassmannian (cf. Example \ref{primerchik}), while the generator $e\in\Ch(\bar{Y}_{2r})$
is the first Chern class of the tautological vector bundle on the
projective bundle $\bar{Y}_{2r}\to\bar{Y}_{2r+1}$.
(These generators coincide with the generators for the previous to the maximal orthogonal grassmannian
constructed in \cite[\S2]{Vishik-u-invariant}, the generator $e$ being the generator of the $W$-type and the generators
$e_1,\dots,e_{2r}$ being the generators of the $Z$-type.)

Let now $Y'_{2r}$ be the variety $Y_{2r}$ over a field extension of $F$
such that $h$ is almost hyperbolic but $K$ is still a field.
The subring $\Ch(Y'_{2r})\subset\Ch(\bar{Y}_{2r})$ is generated by $e,e_2,\dots,e_{2r}$
(all the above generators without $e_1$).
In particular, the $\Ch(Y'_{2r})$-module $\Ch(\bar{Y}_{2r})$ is generated
by two elements: $1$ and $e_1$.
The norm map $\Ch(\bar{Y}_{2r})\to\Ch(Y'_{2r})$ is the homomorphism of
$\Ch(Y'_{2r})$-modules mapping $1$ to $0$ and $e_1$ to $e$
(because the only possibility for the conjugate of $e_1$ is the element $e+e_1$).
In particular, the image of the norm map is the ideal generated by $e$.
The ring $\Ch(Y'_{2r})/N$ is generated by all $e_i$ with $i\geq2$ subject to the
relations $e_i^2=e_{2i}$.
The subring $\BCh(Y_{2r})/N\subset\Ch(Y'_{2r})/N$ contains the elements
$e_2,e_4,\dots,e_{2r}$.
In the case of {\em generic} $h$, the subring $\BCh(Y_{2r})/N$ does not contain any $e_i$
with odd $i$:
otherwise the canonical dimension of $Y_{2r}$ (and therefore of $X_r$) would be smaller than
$$
\dim Y_{2r}-(2+4+\dots+2r)=r(r+2)=\dim X_r.
$$
It follows that the subring $\BCh(Y_{2r})/N$ is generated by the elements
$e_2,e_4,\dots,e_{2r}$.
In particular, the only non-zero homogeneous element of dimension $\dim X_r$ in
$\BCh(Y_{2r})/N$ is the product $e_2e_4\dots e_{2r}$ and therefore
$$
[X_r]=e_2e_4\dots e_{2r}\in\Ch(Y'_{2r})/N.
$$
The annihilator of $[X_r]$ in $\Ch(Y'_{2r})/N$ is therefore the ideal
generated by $e_2,\dots,e_{2r}$ and it follows that for $n=2r+1$ and
{\em almost hyperbolic} $h$ the ring $\Ch(X_r)/N$ is generated by elements
$e_3,e_5,\dots,e_{2r-1}$ subject to relations $e_i^2=0$.
\end{example}


\def\cprime{$'$}

\end{document}